\pdfoutput=1
\documentclass{article}

\usepackage{enumerate,amsmath,amsfonts,amsthm,amssymb,graphicx, float, tikz, color, xcolor}
\usetikzlibrary{patterns}

\def\hn1{H_{n+1}}
\def\dpda{\frac{\Delta p}{\Delta a}}
\def\calh{{\cal H}}
\def\hcap{\calh\cap\hn1}
\def\hcup{\calh\cup\hn1}

\def\arccot{\cot^{-1}}
\newcommand{\be}{\begin{enumerate}}
\newcommand{\ee}{\end{enumerate}}
\newcommand{\bi}{\begin{itemize}}
\newcommand{\ei}{\end{itemize}}

\newcommand{\rr}{\mathbb{R}}

\newtheorem{theorem}{Theorem}

\newtheorem{lemma}[theorem]{Lemma}

\newtheorem{corollary}[theorem]{Corollary}

\theoremstyle{definition}
\newtheorem{example}[theorem]{Example}

\theoremstyle{definition}

\newtheorem*{PAC}{Keleti's Perimeter to Area Conjecture}
\newtheorem*{gyenesexample}{Gyenes's Example}
\newtheorem*{GyenesMain}{Gyenes's Polyhedral Theorem}

\title{Bounded -- Yes, but  4?}

\author{Paul D. Humke, Cameron Marcott, Bjorn Mellem, Cole Stiegler}
\date{November 11, 2013}

\begin{document}
\maketitle
\thispagestyle{empty}

\begin{abstract}\label{sec:abstract}
In this paper we will examine unions of oriented and non-oriented unit squares in same plane and measure the ratio of perimeter to area of these unions.  In 1998, T. Keleti published the conjecture that this ratio never exceeds 4. We outline the current state of research on this conjecture and give two proofs of a special case. Finally, we explore the difficulties that arise from using similar methods in the general case and examine properties of any potential counterexample.
\end{abstract}

\section{Introduction} \label{sec:introduction}
The purpose of this paper is to introduce T\'amas Keleti's {\it{Perimeter to Area Conjecture}}, {\it
{PAC}}. This conjecture concerning unions of unit squares in the Euclidean plane, $\mathbb 
R^2$, is at the same time, easy and elementary to state, but elusive to penetrate. As with other 
such {\it elementary} statements, the fact that it has defied solution shows that we don't 
understand something quite fundamental about basic geometry in a very familiar ambient 
setting. Here we outline the current state of knowledge about the problem and  in Section \ref{sec:counterexample} use the isoperimetric inequality to provide some new insight concerning a potential counterexample. The conjecture itself is this:

\begin{PAC} \label{conj:PAC}
{\it The perimeter to area ratio of the union of finitely many unit squares in a plane does not exceed 4.}
\end{PAC}

The problem of showing that this ratio is simply bounded at all let alone by 4 first 
seems to have appeared as {\it{Problem 6}} on the famous Hungarian {\it
{Schweitzer Competition}} in 1998 \cite{S1988}. Even that is not 
completely trivial, though several Hungarian undergraduates managed a proof 
for the competition. Later that same year, Keleti published his {\it{Perimeter to 
Area Conjecture}} that this bound is actually 4. To date, the best known bound is 
slightly less than 5.6. This bound was found by Keleti's student Zolt\'an Gyenes in 
his master's thesis \cite{GT}. The {\it{PAC}} is particularly intriguing as 
some of its obvious generalizations are false. In particular, Gyenes, also in 
\cite{GT}, showed that a corresponding conjecture is false if ``square'' 
is replaced by ``convex set'' even if the union consists of two congruent convex 
sets. The example is disarmingly simple. Let $E_1$ denote a unit square centered at the origin with a 
small iscoseles triangle ``clipped'' from one corner, see Figure \ref{clipped} below.
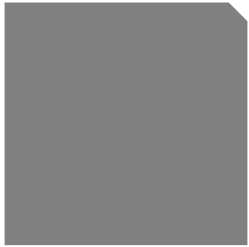
\begin{figure}[H]\label{clipped}
\begin{center}
\scalebox{.8}{

\begin{tikzpicture}[color = gray, line width = 1pt]
\draw[fill=gray] (-2,-2)  -- (2,-2) -- (2,1.7)--(1.7,2) -- (-2,2) -- cycle;
\end{tikzpicture}
}
\end{center}
\end{figure}
 If $x$ denotes the height of that clipped triangle, then in terms of $x$,  the perimeter and area of 
 $E_1$ are respectively $p(x)=4-x(2-\sqrt{2}) \text{ and }a(x)= 1-x^2/2$. From 
 this it is easy to compute that the derivative of the perimeter to area ratio is 
 negative when $x=0$ and so for small values of $x$, the perimeter to area ratio 
 of $E_1$ is less than $4$. But if $E_2$ denotes the set obtained by rotating 
 $E_1$ by $\pi$ about the origin, then $E_1\cup E_2$ is simply the original unit 
 square whose perimeter to area ratio is obviously exactly $4$. This is the example referred to below.

\begin{gyenesexample} \label{ex:gyenes_counter_example}
{\it{There exist congruent convex sets , $E_1 \cong E_2\subset \rr^2$ such that the perimeter to area ratio  
for $E_1\cup E_2$ exceeds the perimeter to area ratio for either one of them.}}
\end{gyenesexample}

While the {\it PAC} as stated remains unresolved, a special case of it is known.

\begin{theorem} \label{thm:oriented_squares}
The perimeter to area ratio of the union of finitely many axis oriented unit squares in a plane does not exceed 4.
\end{theorem}

Gyenes gives a proof of this theorem in \cite{GT}. In Section \ref
{sec:state_of_conjecture} we 
outline Gyenes's proof of Theorem \ref{thm:oriented_squares} and show how he obtains 
the bound of 5.6 for the 
general case. Two additional elementary proofs for Theorem \ref{thm:oriented_squares} are 
given in Sections \ref{sec:bump_method} and \ref{sec:boundary_strips}. 
In Section \ref{sec:centered_squares} we examine the special case of non-
oriented squares centered at the 
same point.  In Section \ref{sec:counterexample} we use the isoperimetric 
inequality to find strict conditions which any optimal counterexample must 
satisfy.  

\section{Notation} \label{sec:notation}
\noindent
Throughout this paper, $\calh = \bigcup_{i=1}^{n} H_i$ will be the union of finitely 
many unit 
squares $H_i$ in $\rr^2$. By a {\it vertex} of $\calh$ we mean a point $p$ on the 
boundary of 
$\calh$ which lies on the boundary of more than one of the $H_i$. The perimeter 
function 
$p(\cdot)$ 
takes a closed, bounded polygonal figure in the plane as input and returns that figure's 
perimeter.  The area function $a(\cdot)$ takes a closed, 
bounded polygonal figure in the plane as input and returns that figure's area.  The value $
\Delta p$ refers to the change in perimeter under a 
given action (adding a square, removing a square, moving a square, etc.).  When adding a 
square $H_{n+1}$, $\Delta p = p(H \cup H_{n+1}) - p(H)$.  
When subtracting the square $H_n$, that is $\Delta p = p(H) - p(H \setminus H_n )$, the 
function  $\Delta a$ is defined analogously.  Finally, if 
$A$ is a region in $\mathbb R^2$, then $\partial A$ will denote the  boundary of $A$. Other 
notation will be defined as needed.

\section{Gyenes's {\em PAC} Results} \label{sec:state_of_conjecture}
\noindent
That the perimeter--area ratio for squares is  bounded by 5.6 and that the bound is  exactly 4 in the special 
case of axis oriented squares both follow from a general theorem obtained by Gyenes in \cite{GT} 
on the surface-area to volume ratio  of the union of finitely many copies of a fixed set in $\rr^n$. Without 
giving details or even complete background definitions, we state this result below to highlight the nature of 
the general theorem. Here, $T_{A,\mu}$ is a measure of the thinness of the set $A$ as measured via a fixed 
probability measure $\mu$; see \cite[Section 2]{GT} for details.

\begin{GyenesMain} {\it If $\cal H$ is the union of a finite set $H_i$ of congruent polyhedra in $\rr^n$, then for any fixed probability measure $\mu$,  the ratio of the surface-area of  $\cal H$ to the volume of $\cal H$ does not exceed $\frac{1}{T_\mu}$, where $T_\mu$ is the infimum of the set $\{T_{A,\mu}\; :\; A\in H_i \}$.}
\end{GyenesMain}

Stripping the broader theorem of its generality, we first present Gyenes's proof of 
Theorem \ref{thm:oriented_squares} and we then show how his 5.6 bound is 
obtained for non-oriented squares.

\begin{proof}[Gyenes's Proof of Theorem \ref{thm:oriented_squares}] \label{proof:gyenes_oriented}
Let $\{H_i\;:\; i=1,2,\dots n\}$ be a finite collection of unit squares in $\rr^2$, suppose that the edges of each 
$H_i$ are either vertical or horizontal, 
and set ${\cal H}=\cup_{i=1}^n H_i$.  To begin, fix $i$ and a non-vertex  point $p\in\partial H_i$.  Let
$\Theta\equiv  \left\{0,\frac{\pi}{2},\pi,\frac{3\pi}{2}\right\}$, For each $\theta\in\Theta$, let $l_{p,H_i}(\theta)$ 
denote the length of the 
line segment in the interior of $H_i$ that begins at $p$ and is in direction $\theta$.  In the case of a single 
oriented unit square, this length is 1 in the 
direction that goes directly across the square and 0 in the other three cardinal directions.
Hence, for a fixed square $H_i$ and fixed non-vertex point $p \in \partial H$, the 
sum of $l_{p,H_i}(\theta)$ over the four cardinal directions is simply
\begin{equation}\label{cardsum}
\sum_{\theta\in\Theta} l_{p,H_i}\left(\theta\right) = 1.
\end{equation}

Now, consider the entire figure $\cal H$ and partition $\partial\cal H$ into finitely many line segments $s_j$ such that 
\begin{enumerate}
\item each $s_j$ is contained entirely in the boundary a single $H_i\equiv H_{i(j)}$, 
\item the $s_j$'s are disjoint except for possibly at their endpoints, and
\item $\bigcup_j s_j = \partial H$. 
\end{enumerate}
Let $\left|s_j\right|$ denote the length of segment $s_j$ and fix $\theta\in\Theta$. 
We bound 
the area of $\cal H$ below as the sum of the areas of the rectangular strips having one 
side $s_j$ (and the other side either 1 or 0), to obtain:

$$
a({\cal H})\ge \sum_j \left|s_j\right| l_{M_j, H_i}(\theta),
$$
where $M_j$ is the midpoint of $s_j$ and $H_i=H_{i(j)}$ is the square 
with $s_j \subset \partial H_i$.

\medskip
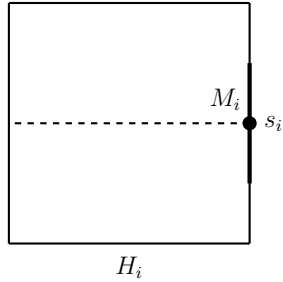
\begin{figure}[H]
\begin{center}
\scalebox{.8}{
$$
\begin{tikzpicture}[color = black, line width = 1pt]
\draw (-2,-2) -- (-2,2);
\draw (-2,2) -- (2,2);
\draw (2,2) -- (2,-2);
\draw (2,-2) -- (-2,-2);

\draw[color = black, line width = 2pt] (2,-1) -- (2,1);
\draw[dashed] (2,0) -- (-2,0);
\draw[color = black, fill = black] (2,0) circle (.1cm);
\draw[color = black] (2.4,0) node {\large{$s_i$}};
\draw[color = black] (1.6,.4) node {\large{$M_i$}};
\draw (0,-2.4) node {\large{$H_i$}};
\end{tikzpicture}
$$
}
\caption{The Geometry in $H_i$ with $\theta=\frac{3\pi}{2}$}
\label{fig:geometry_in_H}
\end{center}
\end{figure}

\noindent
Averaging across the four cardinal directions and using (\ref{cardsum}) then yields:

\begin{equation*}
a({\cal H})\ge\frac{1}{4}\sum_{\theta\in\Theta} \sum_j \left|s_j\right| l_{M_j, H_i}(\theta)
=\frac{1}{4} \sum_j \left|s_j\right|\sum_{\theta\in\Theta} l_{M_j,H_i}(\theta)= \frac{1}{4} \sum_i \left|s_i\right|.
\end{equation*}
Hence, $\frac{1}{4} \sum_i \left|s_i\right| \leq a(H)$.  However, $\sum_i\left|s_i\right|$ is exactly $p({\cal H})$ and therefore, $\frac{1}{4} p({\cal H}) \leq a({\cal H})$ and $\frac{p({\cal H})}{a({\cal H})} \leq 4$, as desired.
\end{proof}

Gyenes obtains his bound of 5.6 for the case of non-oriented squares using an area finding integral and Fubini's Theorem. Again, we present the Gyenes proof, but restrict the scope to the squares in $\rr^2$

\begin{proof}[The Gyenes Bound]
Let $\{H_i\;:\; i=1,2,\dots n\}$ and ${\cal H}=\cup_{i=1}^n H_i$ be as in the previous proof and let $\theta\in[0,2\pi)$ be fixed. Define the {\it thickness} in the direction $\theta \in [0,2\pi)$ at a  point $p \in \partial \cal H$  to be

$$
\tau(p,\theta) := l_{p,H}(\theta)\sin(\phi)
$$

\noindent
where $l_{p,\calh}(\theta)$ is the length of the line segment transversing the interior of $\calh$ that begins at $p$ and is in direction $\theta$. The angle 
$\phi=\phi(p,\theta)$ is the smaller of the two angles that this line segment  makes with $\partial \calh$ at $p$ and is well defined except at vertices of $\cal H$. See Figure \ref{fig:thickness_at_point}.

\begin{figure}[H] 
\begin{center}
\scalebox{.8}{
$$\begin{tikzpicture}[color = black, line width = 1pt]

\draw[black, shift={(-1.5 cm,2.1 cm)},rotate=-20] (-2,-2) rectangle (2,2); 
\draw[black, rotate around={0:(0,0)}] (-2,-2) rectangle (2,2); 

\draw[dashed] (2,-1.4641) -- (-1.75,4.3);

\draw[color = black, fill = black] (2,-1.4641) circle (.1cm);

\node at (2.4,-1.4641) {$p$};
\node at (-.2,3) {$ l_{p,H}(\theta)$};

\draw[dashed, ->] (5.5,-2) -- (4.5,-.26795);

\draw (2,-.7641) arc (90:120:.7);
\draw[<->] (4.5,-2) -- (7,-2);
\draw (6,-2) arc (0:120:.5);

\node at (1.75,-.5) {$\phi$};
\node at (5.85,-1.3) {$\theta$};

\end{tikzpicture}$$}
\caption{Thickness at a Boundary Point}
\label{fig:thickness_at_point}
\end{center}
\end{figure}
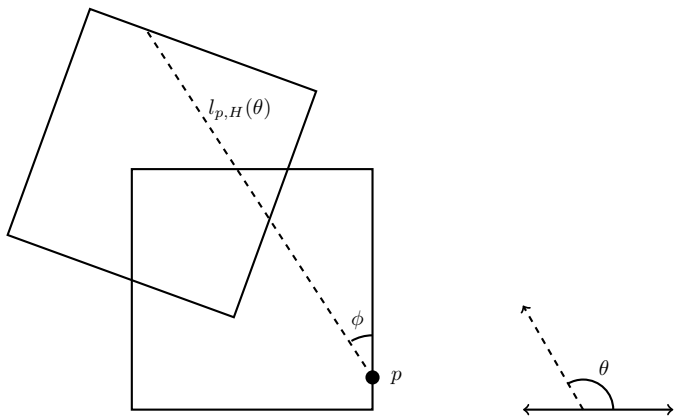

As $\calh$ is the finite union of unit squares, it follows that for a fixed $\theta$, the function $ l_{p,H}(\theta)$ is well defined on $\partial\calh$ and  is piecewise linear. Hence, $\tau(p,\theta)$ is well defined and piecewise linear except at the vertices of $\calh$. The projection of $\calh$ onto a line of direction $\theta+\frac{\pi}{2}$ is a finite union of non-degenerate closed intervals, and  hence, the area of $\calh$, $a(\calh)$ is the (Riemann) integral of $ l_{p,H}(\theta)$ over that projection.  Changing variables then to integrate around the boundary of $\calh$, we conclude that for each fixed $\theta\in[0,2\pi)$

$$
a(\calh)=\int_{p\in\partial H} \tau(p,\theta)\, ds.
$$
where $s$ is the arc-length parameterization of $\partial\calh$. Averaging over $[0,2\pi]$ and applying Fubini yields

\begin{align}\label{eqn:fubini}
a(H) &=\frac{1}{2\pi} \int_0^{2\pi}\int_{p\in\partial H} \tau(p,\theta)\, ds\, d\theta \notag \\
&=\frac{1}{2\pi} \int_{p\in\partial H}\int_0^{2\pi} \tau(p,\theta)\, d\theta\, ds.
\end{align}

Each non-vertex point, $p \in \partial\calh$ is also on the boundary of a unique component unit 
square, $H_{i(p)}$.  We let $l^*_{p,\calh}(\theta)=l_{p,H_{i(p)}}(\theta) \leq l_{p,\calh}(\theta)$ 
and $\tau(p,\theta) := l^*_{p,H}(\theta)\sin(\phi)$. Since $\sin(\phi)\ge0$, it follows that 
$\tau(p,\theta) \ge  \tau^*(p,\theta)$ for all non-vertex points $p\in\partial\calh$. See Figure \ref{fig:less_thickness_at_point}.

\begin{figure}[H] 
\begin{center}
\scalebox{.8}{
$$\begin{tikzpicture}[color = black, line width = 1pt]

\draw[black, shift={(-1.5 cm,2.1 cm)},rotate=-20] (-2,-2) rectangle (2,2); 
\draw[black, rotate around={0:(0,0)}] (-2,-2) rectangle (2,2); 

\draw[dashed] (2,-1.4641) -- (-.34,2);

\draw[color = black, fill = black] (2,-1.4641) circle (.1cm);

\node at (2.4,-1.4641) {$p$};
\node at (1.2,.8) {$ l^*_{p,H}(\theta)$};

\draw[dashed, ->] (5.5,-2) -- (4.5,-.26795);

\draw (2,-.7641) arc (90:120:.7);
\draw[<->] (4.5,-2) -- (7,-2);
\draw (6,-2) arc (0:120:.5);

\node at (1.75,-.5) {$\phi$};
\node at (5.85,-1.3) {$\theta$};

\end{tikzpicture}$$}
\caption{Less Thickness at a Boundary Point}
\label{fig:less_thickness_at_point}
\end{center}
\end{figure}

\noindent 
Now let  $T^*=\inf \int_0^{2\pi} \tau^*(p,\theta)\, d\theta $ 
where the $\inf$ is taken over the non-vertex points of $H_{i(p)}$. Then, 
$T^*$ is independent of $p$ (or $i(p)$) and so combining the 
inequality $\tau\ge\tau^*$ with the definition of $T^*$ and equation  (\ref{eqn:fubini}) we obtain:
\begin{align}
a(H) &=\frac{1}{2\pi} \int_{p\in\partial \calh}\int_0^{2\pi} \tau(p,\theta)\, d\theta\, ds \notag \\
&\ge\frac{1}{2\pi} \int_{p\in\partial \calh}\int_0^{2\pi} \tau^*(p,\theta)\, d\theta\, ds \notag \\
&\ge \frac{1}{2\pi} \int_{p\in\partial \calh} T^*=\frac{T^*}{2\pi}\cdot p(\calh). 
\end{align}
Or,
\begin{equation}
\frac{p(\calh)}{a(\calh)}\le \frac{2\pi}{T^*}.
\end{equation}
Finding $T^*$ is just a matter of computation. First,
$$
l^*(p,\theta) =
\begin{cases} 
(1-x) \sec(\theta) & 0 \leq \theta \leq \arccot (1-x)\\
\csc(\theta) &  \arccot (1-x)\le\theta\le\pi-\arccot(x)\\
-x\sec(\theta) & \pi-\arccot(x)\le\theta\le\pi\\
0 & \pi \leq x < 2\pi.
\end{cases}
$$
See Figure \ref{fig:l*(p,t)}.

\begin{figure}[H] 
\begin{center}
\scalebox{}{
$$\begin{tikzpicture}[color = black, line width = 1pt]

\draw[black] (-2,-2) rectangle (2,2); 

\node at (-1,-2.4) {$p=(x,0)$};
\draw[color = black, fill = black] (-1,-2) circle (.1cm);
\node at (3.6,-1.2) {$\left(1,(1-x)\tan(\theta)\right)$};
\draw[color = black, fill = black] (2,-1) circle (.1cm);
\draw[<->] (-1,-2)--(2,-1);
\node at (1,2.3) {$\left(x+\cot(\theta),1\right)$};
\draw[color = black, fill = black] (1,2) circle (.1cm);
\draw[<->] (-1,-2)--(1,2);
\node at (-3.3,-1) {$\left(0,-x\tan(\theta)\right)$};
\draw[color = black, fill = black] (-2,-1) circle (.1cm);
\draw[<->] (-1,-2)--(-2,-1);

\draw[dashed] (-1,-2) -- (2,2);
\draw[dashed] (-1,-2) -- (-2,2);

\end{tikzpicture}$$}
\caption{Computing $l^*(p,\theta)$}
\label{fig:l*(p,t)}
\end{center}
\end{figure}

In this normalized setting $\theta=\phi$ so that

$$
\tau^*(p,\theta) =
\begin{cases} 
(1-x) \tan(\theta) & 0 \leq \theta \leq \arccot (1-x)\\
1 &  \arccot (1-x)\le\theta\le\pi-\arccot(x)\\
-x\tan(\theta) & \pi-\arccot(x)\le\theta\le\pi\\
0 & \pi \leq x < 2\pi,
\end{cases}
$$
and hence,

$$
\int_0^{2\pi} \tau^*(p,\theta)\, d\theta\ = 
\frac{\ln(1+(1-x)^2)}{2}-\ln(1-x)+\pi-\arccot(1-x)-\arccot(x).
$$

This function is increasing over $[0,1]$ so that the minimum is $\frac{2\pi}{\frac{1}{2}\ln 2 + \frac{\pi}{4}}$ occurring at $x = 0$. Hence, $\frac{p(\calh)}{a(\calh)} \leq \frac{2\pi}{\frac{1}{2}\ln 2 + \frac{\pi}{4}} \leq 5.6$.
\end{proof}
Details of the more general theorem can be found in \cite{GT}.

\section{The Bump Method}\label{sec:bump_method}

In this section we present the first of two additional proofs of Theorem \ref{thm:oriented_squares}. This proof relies on the following elementary fact about rectangles.

\begin{lemma}\label{rectangleineq} 
Let $R\subset[0,1]^2$ be a rectangle.Then $\frac{p(R)}{a(R)}\ge4$.
\end{lemma}

\begin{proof}
If $R$ is an ``a" by ``b"" rectangle, then as $R\subset[0,1]^2$, $0\le a,b\le 1$. Hence, 
$a(b-1)+b(a-1)\le 0$, and the desired inequality immediately follows.
\end{proof}

\begin{proof}[Proof of Theorem \ref{thm:oriented_squares} - The Bump Method]

Suppose that $\calh=\cup_{i=1}^nH_i$ where $H_i$ is an oriented unit square 
(vertical and horizontal edges) for $i=1,2,\dots n$. We induct on the number of 
squares in $\calh$. Clearly, $\frac{p(\calh)}{a(\calh)} = 4$ if $\calh$ is a single 
square. Suppose that $\calh$ consists of $n$ squares and that $\frac{p(\calh)}{a
(\calh)} \leq 4$.  We show that adding another square will not cause the ratio to 
exceed 4.  That is, no matter how the square $H_{n+1}$ is added to the figure,

$$\frac{p(\calh \cup H_{n+1})}{a(\calh \cup H_{n+1})}\leq 4$$

\noindent Note that
$$
\frac{p(\calh \cup H_{n+1})}{a(\calh \cup H_{n+1})} = \frac{p(\calh)+\Delta p}{a(\calh)+\Delta a},
$$
and because $\frac{p(\calh)}{a(\calh)} \leq 4$, it is sufficient to show that 
$\frac{\Delta p}{\Delta a} \leq 4$. To calculate $\frac{\Delta p}{\Delta a}$, one
only needs to examine how $H_{n+1}$ intersects the original figure. We identify at most four 
disjoint rectangles in $\hn1$ that enable us to use Lemma \ref{rectangleineq} 
to finish the proof.

If $\calh\cap \hn1=\emptyset$ (or is a single point), then $\frac{\Delta p}{\Delta a}=\frac{p(\hn1)}{a(\hn1)}=4$ and we are finished.

Hence, we may assume that $\calh\cap \text{interior}(\hn1)\not=\emptyset.$  Since $
\calh$ is comprised of oriented unit squares, and $\calh\cap\hn1\not=\emptyset$ 
it follows that $\text{interior}(\calh)$ contains at least one vertex of $\hn1$.  There 
are several cases depending on the 
number of vertices of $\hn1$ that are in $\calh$.

\begin{enumerate}
\item[Case 1] $\hcap$ contains exactly one vertex of $\hn1$.

We suppose that the bottom left vertex of $\hn1$, 
denoted $v$ is at the origin and that $v\in\calh$.   
Suppose $(x_r,y_r)$ and $(x_s,y_s)$ are the respective rightmost and topmost points of 
$\calh \cap H_{n+1}$.

\smallskip
We ``bump out" $\calh\cap\hn1$ by replacing $\calh\cap\hn1$  with the rectangle $RB$  
whose vertices are $v=(0,0),\ (x_r,0),\ (0,y_s),\ \text{ and }(x_r,y_s)$. See Figure \ref{fig:bumping_out}.

\begin{figure}[H]
$$
\scalebox{.6}{
\begin{tikzpicture}[color = black, line width = 1pt]

\draw (-2,-2) rectangle (4,4);
\node at (3.4,4.5) {\Large{$H_{n+1}$}};

\draw[pattern=north west lines] (0,2) rectangle (3,0);
\draw[dashed] (0,2) -- (3,2);
\draw[dashed] (3,2) -- (3,0);

\draw (0,0) -- (0,2);
\draw (0,0) -- (3,0);
\draw (0,2) -- (-4,2);
\draw (3,0) -- (3,-4);
\draw[color = black, fill = black] (-2,2) circle (.1cm);
\draw[color = black, fill = black] (3,-2) circle (.1cm);
\node at (-4.5,2) {\Large{$\calh$}};
\node at (-1.4,2.4) {\Large$(0,y_s)$};
\node at (2.4,-1.6) {\Large$(x_r,0)$};

\end{tikzpicture}
}
$$
\caption{``Bumping Out'' within $H_{n+1}$}
\label{fig:bumping_out}
\end{figure}
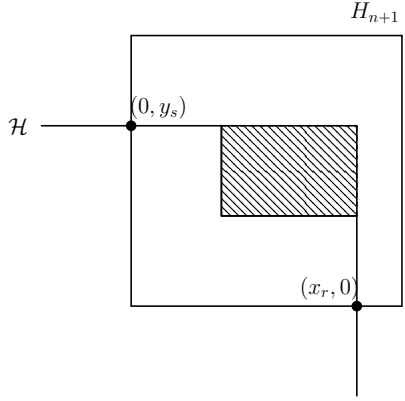

For notational simplicity, denote the ``stair-step" region $\hcap$ by $SS$. Denote the 
portion of the $\partial (SS)$ which also lies on the boundary of $\hn1$ by $\gamma_1$ 
and the remainder of $\partial (SS)$ by$\gamma_2$. Then,

\begin{align*}
\dpda &=\frac{p(\hcup)-p(\calh)}{a(\hcup)-a(\calh)} = 
\frac{\big(p(\calh)-\ell(\gamma_2)+4-\ell(\gamma_1)\big)-p(\calh)}{\big(a(\calh)-a(SS)+1\big)-a(\calh)}\\
&=\frac{4-\ell(\gamma_1)-\ell(\gamma_2)}{1-a(SS)}
=\frac{4-p(RB)}{1-a(SS)}
\le\frac{4-p(RB)}{1-a(RB)}
\end{align*}
By Lemma \ref{rectangleineq}, $\frac{p(RB)}{a(RB)}\ge 4$ and hence, $\frac{4-p(RB)}{1-a(RB)}\le 4$ completing the proof.

\item[Case 2] $\hcap$ contains exactly two vertices, $v_1$ and $v_2$ of $\hn1$.

We first ``bump out" $\hcap$ as 
in Case 1, but this time we obtain two rectanglular regions, $RB_1$ at $v_1$ 
and $RB_2$ at $v_2$. If $RB_1\cap RB_2=\emptyset$, then using analogous estimates as in Case 1 we find that
$$
\dpda =\frac{p(\hcup)-p(\calh)}{a(\hcup)-a(\calh)} 
\le\frac{4-p(RB_1)-p(RB_2)}{1-a(RB_1)-a(RB_2)}.
$$
Now, from Lemma \ref{rectangleineq} we know that both $\frac{p(RB_1)}{a
(RB_1)}\ge 4$ and $\frac{p(RB_2)}{a(RB_2)}\ge 4$ and hence, it follows directly 
that $\frac{4-p(RB_1)-p(RB_2)}{1-a(RB_1)-a(RB_2)}\le4$ as desired. 

\smallskip
If  $RB_1\cap RB_2\not=\emptyset$, there are two subcases to consider 
depending on whether $v_1$ and $v_2$ are adjacent or diagonally opposite one 
another. 

\smallskip
In the case that $v_1$ and $v_2$ are adjacent and $RB_1$ overlaps $RB_2$, we ``bump out" the region $RB_1\cup RB_2$ to the smallest oriented rectangle 
containing $RB_1\cup RB_2$ and denote that rectangle by $RB^*$. See Figure \ref{bumpout2}. The computation we did in Case 1 now applies to this situation and we compute that
$$
\dpda =\frac{p(\hcup)-p(\calh)}{a(\hcup)-a(\calh)} 
\le\frac{4-p(RB^*)}{1-a(RB^*)}.
$$
An application of Lemma \ref{rectangleineq} completes this case. 

\begin{figure}[H]
$$
\scalebox{.6}{
\begin{tikzpicture}[color = black, line width = 1pt]

\draw[black] (-4,-4) rectangle (4,4); 
\node at (3.4,4.5) {\Large{$H_{n+1}$}};

\draw[pattern=north west lines] (0,2) rectangle (4,0);
\draw[dashed] (0,2) -- (4,2);

\draw (0,-5) -- (0,2);
\draw (-3,0) -- (5,0);
\draw (0,2) -- (-5,2);
\draw (-3,0) -- (-3,-5);
\end{tikzpicture}
}
$$
\caption{Another ``Bumping Out'' within $H_{n+1}$}
\label{bumpout2}
\end{figure}
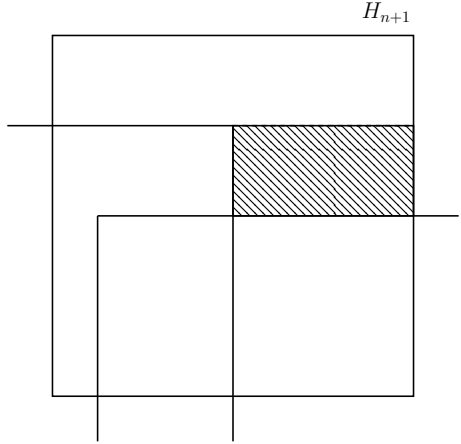

\smallskip
Finally, if $v_1$ and $v_2$ are diagonally opposite and $RB_1$ overlaps $RB_2$, then $\Delta p \leq 0$ so that the conclusion trivially holds. 

\end{enumerate}
This, then completes Case 2. The cases where $\hcup$ contains exactly three or exactly four points are completely analogous, and hence this also completes the ``Bump Method" proof of Theorem \ref{thm:oriented_squares}. 
\end{proof}

\section{The Method of Boundary Strips} \label{sec:boundary_strips}

\begin{proof}[Proof 3 of Theorem \ref{thm:oriented_squares} - Boundary Strips]
This proof is similar to Gyenes's in Section \ref{sec:state_of_conjecture} in that 
we add up the area of strips along the boundary. However, we apply the idea 
inductively, again showing that $\frac{\Delta p}{\Delta a} \leq 4$ when adding a 
square to $\calh$.

To begin, we subdivide the boundary of $H_{n+1}$ into (necessarily finitely many) 
nonoverlapping line segments of three types.  
\begin{enumerate}
\item[$P_0$]

These segments (possibly degenerate) are maximal subsegments of $\partial H_{n+1} \cap \calh$. Two may intersect, but only at a vertex of $\hn1$. 
 
\item[$P_1$]
 
A maximal segment $S\subset\partial H_{n+1} \setminus \calh$ is a $P_1$ segment if any
line that intersects $S$ and is orthogonal to $S$ also intersects $\hcap$.

\item[$P_2$]

These are the remaining maximal segments  $S\subset\partial H_{n+1} \setminus \calh$
and are characterized by the property that any line that intersects such a segment and is orthogonal 
to it misses $\hcap$.
\end{enumerate}

Note that the mirror image of a $P_2$ segment is another $P_2$ segment, while the mirror image of  a $P_1$ segment is always contained in a $P_2$ segment. The mirror image of a point in a $P_0$ segment can be either a $P_0$ point or a $P_1$ point. See Figure \ref{bseg}.

\begin{figure}[H]\label{bseg}
\begin{center}
\scalebox{.6}{
$$
$$
\begin{tikzpicture}[color = black, line width = 1pt]

\draw[black] (-2,-2) rectangle (2,2); 
\node at (0,0) {\large{$H_{n+1}$}};

\draw (-3,1) -- (-1,1);
\draw (-1,1) -- (-1,-1);
\draw (-1,-1) -- (1,-1);
\draw (1,-1) -- (1,-3);
\node at (-3,.5) {\Large{$\calh$}};

\draw[color = black, line width = 2pt] (-2,-2) -- (-2,1);
\draw[color = black, line width = 2pt] (-2,-2) -- (1,-2);
\node at (0,-2.5) {{$P_0$}};
\node at (-2.5,-.5) {{$P_0$}};
\draw[color = black, fill = black] (-2,-2) circle (.1cm);
\draw[color = black, fill = black] (-2,1) circle (.1cm);
\draw[color = black, fill = black] (1,-2) circle (.1cm);

\end{tikzpicture}
$$$$
\hspace{5mm}
$$
\begin{tikzpicture}[color = black, line width = 1pt]

\draw[black] (-2,-2) rectangle (2,2); 
\node at (0,0) {\large{$H_{n+1}$}};

\draw (-3,1) -- (-1,1);
\draw (-1,1) -- (-1,-1);
\draw (-1,-1) -- (1,-1);
\draw (1,-1) -- (1,-3);
\node at (-3,.5) {\Large{$\calh$}};

\draw[color = black, line width = 2pt] (-2,2) -- (1,2);
\draw[color = black, line width = 2pt] (2,1) -- (2,-2);
\node at (-.5,2.3) {{$P_1$}};
\node at (2.3,-.5) {{$P_1$}};
\draw[color = black, fill = black] (-2,2) circle (.1cm);
\draw[color = black, fill = black] (1,2) circle (.1cm);
\draw[color = black, fill = black] (2,1) circle (.1cm);
\draw[color = black, fill = black] (2,-2) circle (.1cm);

\end{tikzpicture}$$$$
\hspace{5mm}
$$
\begin{tikzpicture}[color = black, line width = 1pt]

\draw[black] (-2,-2) rectangle (2,2); 
\node at (0,0) {\large{$H_{n+1}$}};

\draw (-3,1) -- (-1,1);
\draw (-1,1) -- (-1,-1);
\draw (-1,-1) -- (1,-1);
\draw (1,-1) -- (1,-3);
\node at (-3,.5) {\Large{$\calh$}};

\draw[color = black, line width = 2pt] (2,2) -- (1,2);
\draw[color = black, line width = 2pt] (1,-2) -- (2,-2);
\draw[color = black, line width = 2pt] (-2,2) -- (-2,1);
\draw[color = black, line width = 2pt] (2,2) -- (2,1);
\node at (-2.5,1.5) {{$P_2$}};
\node at (2.5,1.5) {{$P_2$}};
\node at (1.5,2.3) {{$P_2$}};
\node at (1.5,-2.3) {{$P_2$}};
\draw[color = black, fill = black] (2,2) circle (.1cm);
\draw[color = black, fill = black] (1,2) circle (.1cm);
\draw[color = black, fill = black] (2,-2) circle (.1cm);
\draw[color = black, fill = black] (1,-2) circle (.1cm);
\draw[color = black, fill = black] (-2,2) circle (.1cm);
\draw[color = black, fill = black] (-2,1) circle (.1cm);
\draw[color = black, fill = black] (2,1) circle (.1cm);

\end{tikzpicture}
$$
$$}
\end{center}
\caption{Boundary Segments on $H_{n+1}$}
\end{figure}
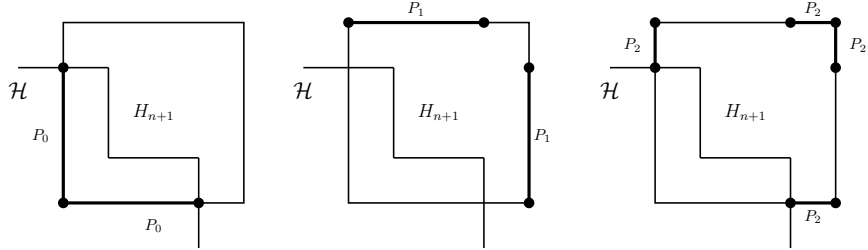

It is clear that no  $P_0$ segment contributes $\Delta p$, but what is also true is that
$P_1$ segments also contribute no net perimeter to $\Delta p$.

To see this, suppose $S$ is a $P_1$ segment and let $S^*$ be the mirror 
image of $S$ on $\partial \hn1$. Then $S^*
\subset P_0$ and consequently, there is a portion of $\partial\calh\cap\hn1$ that covers $S$ in its 
entirety and this portion of the perimeter of $\calh$ is no longer part of the boundary of $\hcup$. 
Thus, although every $P_1$ segment contributes to the boundary of $\hcup$, that addition
is balanced by a subtraction from $\partial\calh$. That is, $P_1$ segments create no net perimeter. 

Hence, only $P_2$ segments contribute net perimeter to $\Delta p$.  Since $P_2$ segments mirror 
each other on $\partial\hn1$,  any pair of $P_2$ segments of individual length $b$ will contribute $2b$ 
toward $\Delta p$. The region between a mirror pair of $P_2$ segments of length $b$ necessarily 
misses $\calh$ and has area $b$.  However, it is not necessarily true that this strip contributes an 
area of $b$ to $\Delta a$ as it is possible to have $P_2$ segments on both the horizontal and vertical 
sides of $H_{n+1}$.  In this case, the horizontal and vertical strips intersect. 
See Figure \ref{fig:intersecting_boundary_strips}.

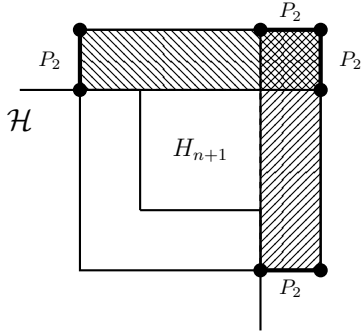
\begin{figure}[H]\label{fig:intersecting_boundary_strips}
\begin{center}
\scalebox{.8}{
$$
\begin{tikzpicture}[color = black, line width = 1pt]

\draw[black] (-2,-2) rectangle (2,2); 
\node at (0,0) {\large{$H_{n+1}$}};

\draw (-3,1) -- (-1,1);
\draw (-1,1) -- (-1,-1);
\draw (-1,-1) -- (1,-1);
\draw (1,-1) -- (1,-3);
\node at (-3,.5) {\Large{$\calh$}};

\draw[color = black, line width = 2pt] (2,2) -- (1,2);
\draw[color = black, line width = 2pt] (1,-2) -- (2,-2);
\draw[color = black, line width = 2pt] (-2,2) -- (-2,1);
\draw[color = black, line width = 2pt] (2,2) -- (2,1);
\node at (-2.5,1.5) {{$P_2$}};
\node at (2.5,1.5) {{$P_2$}};
\node at (1.5,2.3) {{$P_2$}};
\node at (1.5,-2.3) {{$P_2$}};
\draw[color = black, fill = black] (2,2) circle (.1cm);
\draw[color = black, fill = black] (1,2) circle (.1cm);
\draw[color = black, fill = black] (2,-2) circle (.1cm);
\draw[color = black, fill = black] (1,-2) circle (.1cm);
\draw[color = black, fill = black] (-2,2) circle (.1cm);
\draw[color = black, fill = black] (-2,1) circle (.1cm);
\draw[color = black, fill = black] (2,1) circle (.1cm);

\draw[dashed] (-2,1) -- (2,1);
\draw[dashed] (1,2) -- (1,-2);
\draw[pattern=north west lines] (-2,1) rectangle (2,2);
\draw[pattern=north east lines] (1,2) rectangle (2,-2);
\end{tikzpicture}
$$}
\end{center}
\caption{Intersecting Boundary Strips}
\end{figure}

Suppose there are $\alpha$ mirrored pairs of horizontal $P_2$ segments that have (individual) 
lengths $h_1, \dots, h_\alpha$. Likewise, suppose there are $\beta$ pairs of vertical $P_2$ segments 
that have lengths $v_1, \dots, v_\beta$ and set

$$
h = \sum_{i = 1}^{\alpha} h_i \ \ \mbox{   and   } \ \ v = \sum_{i = 1}^{\beta} v_i.
$$

By inclusion/exclusion, the contribution to $\Delta a$ of all of the strips is $h + v - hv$, and so
$$
\Delta a \geq h + v - hv \geq h + v - \left(\frac{h + v}{2}\right)^2.
$$

Since the $P_2$ segments are the only segments making a net positive contribution to $\Delta p$,
$\Delta p \leq 2(h+ v)$ and hence,
\begin{equation}\label{one}
\frac{\Delta p }{\Delta a} \leq \frac{2(h + v)}{h + v - hv}.
\end{equation}

Computing the maximum of (\ref{one}) it is easy matter to see that $\frac{2(h + v)}{h + v - hv} \leq 4$ when $h,v \in (0,1]$. In fact, (\ref{one}) only obtains its maximum value of $4$ when $h = v = 1$. This computation then completes the inductive proof and this section.
\end{proof}

\section{The Non-oriented Case} \label{sec:general_case}

\subsection{Need to Think Outside the Box}
The method of looking at $\frac{\Delta p}{\Delta a}$ when adding an $n+1$st square fails for non-oriented squares. Here is a simple counterexample.

\begin{example}
In the figure below, $H$ intersects $H_{n+1}$ covering all but one small triangle in the bottom left corner of $H_{n+1}$ with base and height $b$.

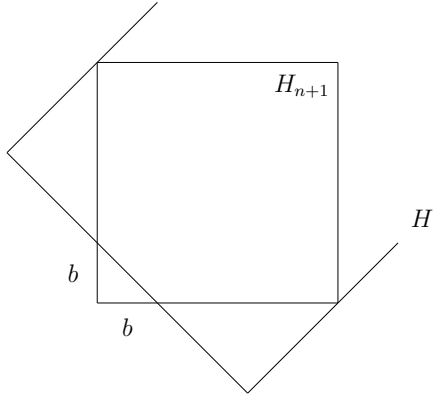
\begin{figure}[H]
\begin{center}
\scalebox{.8}{
$$\begin{tikzpicture}
\draw (-2,-2) -- (-2,2);
\draw (-2,2) -- (2,2);
\draw (2,2) -- (2,-2);
\draw (2,-2) -- (-2,-2);
\draw (1.4,1.6) node {\large{$H_{n+1}$}};

\draw (.5, -3.5) -- (3,-1);
\draw (-3.5, .5) -- (-1,3);
\draw (.5,-3.5) -- (-3.5, .5);

\draw (3.4, -.6) node {\large{$H$}};

\draw (-1.5, -2.4) node {\large{$b$}};
\draw (-2.4, -1.5) node {\large{$b$}};

\end{tikzpicture}$$}
\caption{A Counterexample}
\label{fig:counter_example}
\end{center}
\end{figure}

\noindent
Then, $\frac{\Delta p}{\Delta a}$ is easily computed.

$$\frac{\Delta p}{\Delta a}=\frac{2b-\sqrt{2}b}{b^2/2}=\frac{4-2\sqrt{2}}{b}$$

By selecting $b$ sufficiently small, $\frac{\Delta p}{\Delta a}$ may be made arbitrarily large. 
Therefore, the method of showing $\frac{\Delta p}{\Delta a} \leq 4$ when adding a square fails for 
non-oriented squares.
\end{example}

\subsection{Being Centered Helps}\label{sec:centered_squares}
\noindent

While this subsection does not directly address our problem, its result will be used in the following section. This result and proof are also presented in [\cite{GT}]. However, we are unsure if this paper is the first place this result appears. In any case, our presentation closely follows Z. Gyenes's.

\begin{theorem} \label{thm:centered_squares}
The perimeter to area ratio of $n$ unit squares centered on the same location is 4.
\end{theorem}

\begin{proof}
Let $\calh$ be the union of squares where all component squares share the same center point.
We partition $\calh$ into disjoint triangles by adding line segments between the common center 
and vertices of the $H_i$'s  as well as the vertices created by the intersection of any two of the $H_i
$'s.  The height of any resulting triangle is $\frac{1}{2}$ since the base of any such triangle 
lies on the boundary of a single $H_i$. See Figure \ref{fig:centered_squares}. Summing the areas 
of all these triangles, we find that 
$a(\calh) = \frac{1}{4} p(\calh)$ and thus, $\frac{p(\calh)}{a(\calh)} = 4$.
\end{proof}

\begin{figure}[H]
\begin{center}
\scalebox{.75}{
$$\begin{tikzpicture}
\draw (-2,-2) -- (-0.8284,-2) -- (0,-2.8284) -- (0.8284,-2) -- (2,-2) -- (2,-0.8284) -- (2.8284,0) -- (2,0.8284) -- (2,2) --  (0.8284,2) -- (0,2.8284) -- (-0.8284,2) -- (-2,2) -- (-2, 0.8284) -- (-2.8284,0) -- (-2,-0.8284) -- cycle;

\draw[fill = gray] (0,0) -- (0,2.8284) -- (0.8284,2) -- cycle;

\draw[dashed] (0,0) -- (1.4142,1.4142) -- (0.8284,2);

\node at (1.0071,0.4071) {\large{$\frac{1}{2}$}};
\end{tikzpicture}$$
}
\end{center}

\caption{One Partitioning Triangle in a Figure with 2 Squares}
\label{fig:centered_squares}

\end{figure}
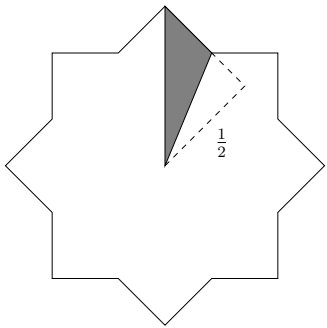

\section{Counterexample?  Pack Your Squares Tightly} \label{sec:counterexample}
\noindent
In this section, we examine the nature of a potential counterexample to the {\em PAC} and show 
that any possible counterexample is constrained in a strong way. We then prove 
a variant of Theorem \ref{thm:oriented_squares} where circles rather than squares are stacked.

Because of the finite nature of the problem, if a counterexample exists there will be a 
counterexample with a minimal number of squares. We focus our attention on this minimal 
counterexample. Throughout this section, $\calh$ will be a finite union of squares such that 
$\frac{p(\calh)}{a(\calh)} > 4$ and for any $i$,

\begin{displaymath}
\frac{p\left(\calh \setminus H_i\right)}{a\left(\calh \setminus H_i\right)} \leq 4.
\end{displaymath}

We refer to such a counterexample as {\it optimal}. The following theorem shows that in such an 
optimal counterexample, any individual component square must share a large portion of its area 
with the rest of the figure.

\begin{theorem} \label{thm:bound}
If $\calh=\bigcup_{i=1}^nH_i$ is an optimal counterexample, then for each $i\le n$, 
$a\left(H_i \cap \left (\calh \setminus H_i\right)\right) > \frac{\pi}{4}$.
\end{theorem}

\begin{proof}
Using the argument of the proof of Theorem \ref{thm:oriented_squares} in 
Section \ref{sec:bump_method}, we assume $\frac{\Delta p}{\Delta a} > 4$ 
when removing any square, $H_i$ from $\calh$.

We maximize $\frac{\Delta p}{\Delta a}$ for a fixed $\Delta a$.  To simplify notation, for a fixed $i$, 
define
$$
\alpha = a\left(H_i \cap \left(\calh \setminus H_i\right)\right)
$$
In the case of removing the square $H_i$, $\Delta a = 1- \alpha$
When we remove the square $H_i$ from $\calh$, a portion of the perimeter of $H_i$ is removed. 
However, some  of this perimeter might have been covered  by $\calh \setminus H_i$ and 
some  perimeter of $\calh\setminus H_i$ that was covered by $H_i$ might be revealed. 
Between the perimeter of $\calh \setminus H_i$ that is revealed and the perimeter of $H_i$ 
that was covered by $\calh$, there 
must be at least enough length to enclose the area that is left behind when $H_i$ is removed.

\smallskip
Thus, $\Delta p \leq 4 - x$ where $x$ is the minimal perimeter required to enclose an area of $
\alpha$ inside an unit square (if the square boundary is used to enclose the area, it contributes
to $x$).

The isoperimetric inequality states that, in general the minimum perimeter needed to enclose a fixed 
area in a plane is given by a circle. This fact holds 
provided a circle of area $\alpha$ can fit inside of the unit square.  It is easily shown that a circle of 
area $\alpha$ can fit inside a unit square so long as $\alpha \leq \frac{\pi}{4}$. Then, assuming $0 
\leq \alpha \leq \frac{\pi}{4}$,
$$
\frac{\Delta p}{\Delta a} \leq \frac{4 - 2\sqrt{\pi \alpha}}{1 - \alpha}
$$
From this it is easy to verify that $\frac{\Delta p}{\Delta a} \leq 4$ when 
$0 \leq \alpha \leq \frac{\pi}{4}$ and hence, $H_i$ must share more than $\frac{\pi}{4}$ 
of its area with $H$, as claimed.
\end{proof}

One generalization of the {\em PAC} is to examine the union of finitely copies of different fixed sets. The 
method just used gives a proof of the {\em PAC} for circles.

\begin{corollary} \label{cor:circles}
The perimeter to area ratio of the union of finitely many circles of radius 1 in the plane does not 
exceed 2.
\end{corollary}

\begin{proof}
Applying the exact same argument as in the proof of Theorem \ref{thm:bound} shows the result.
\end{proof}

\section{Conclusion} \label{sec:conclusion}
Keleti's {\it{Perimeter to Area Conjecture}} is particularly intriguing in a number of ways. First, a full generalization to convex sets is not true so that  any appropriate generalized version must involve a parsing of the convex sets in some, presumably geometric way. But what this might be is a mystery. Second is the simplicity of the conjecture itself. The fact that this conjecture is not settled seems to show that we are missing some important geometric fact concerning unions of square regions. Finally, it is interesting that disparate approaches to the problem can provide new clues. For example, the method of inclusion/exclusion yields the result that the "virgin" area of each square of an optimal counterexample  can be no more than $1-\pi/4$.  

In a subsequent paper, \cite{JAA} we study the differentiability properties of both the perimeter and area functions of a union of n unit squares in the plane.

\vspace{.4in}
\begin{small}
\noindent 
Paul D. Humke, \\
 {\sc  Department of Mathematics, St. Olaf College\\  
Northfield, MN 55057
 email: {\tt humkep@gmail.com}
 }
 
\medskip\noindent 
Cameron Marcott, \\
 {\sc  Department of Mathematics, St. Olaf College\\  
Northfield, MN 55057
 email: {\tt cam.marcott@gmail.com}
 }
 
\medskip\noindent 
Bjorn Mellem, \\
 {\sc  Department of Mathematics, St. Olaf College\\  
Northfield, MN 55057
 email: {\tt contactatrius@hotmail.com}
 }
 
\medskip\noindent 
Cole Stiegler, \\
 {\sc  Department of Mathematics, St. Olaf College\\  
Northfield, MN 55057
 email: {\tt cole.stiegler@gmail.com}
 }

\end{small}

\end{document}